\documentclass{birkjour}

\usepackage{wasysym}
\usepackage{amsmath}
\usepackage{amssymb}
\usepackage{amsthm}

\usepackage{epsfig}  
\usepackage{color}	 %
\input{xy}
\xyoption{poly}
\xyoption{2cell}
\xyoption{all}

\newtheorem{thm}{Theorem}[section]
\newtheorem{prop}[thm]{Proposition}
\newtheorem{lemma}[thm]{Lemma}
\newtheorem{cor}[thm]{Corollary}

\theoremstyle{definition}
\newtheorem{definition}[thm]{Definition}
\newtheorem{example}[thm]{Example}

\theoremstyle{remark}
\newtheorem{remark}[thm]{Remark}

\numberwithin{equation}{section}

\newcommand{\zg}{\gamma}
\newcommand{\zG}{\Gamma}

\newcommand{\zs}{\sigma}
\newcommand{\zS}{\Sigma}
\newcommand{\calc}{\mathcal{C}}

\newcommand{\cale}{\mathcal{E}}
\newcommand{\calj}{\mathcal{J}}

\newcommand{\Hom}{\textup{Hom}}

\newcommand{\Ext}{\textup{Ext}}
\newcommand{\Ann}{\textup{Ann}}

\newcommand{\E}{\Ext^2_C(DC,C)}

\newcommand{\Ctilde}{\widetilde{C}}
\begin{document}
\title[Modules that do not lie on local slices]{Modules over cluster-tilted algebras that do not lie on local slices}
%%
%% Now edit the following to give your name and address:
%% 
\author{Ibrahim Assem}
\address{D\'epartement de Math\'ematiques,
Universit\'e de Sherbrooke,
Sherbrooke, Qu\'ebec,
Canada J1K 2R1}
\email{ibrahim.assem@usherbrooke.ca}

\author{Ralf Schiffler}
\address{Department of Mathematics, University of Connecticut, 
Storrs, CT 06269-3009, USA}
\email{schiffler@math.uconn.edu}

\author{Khrystyna Serhiyenko}\thanks{The first author gratefully acknowledges partial support from the NSERC of Canada. The second author was supported by the NSF CAREER grant DMS-1254567. The third author was  supported by the NSF Postdoctoral fellowship MSPRF-1502881.}
\address{Department of Mathematics, University of California, Berkeley, 
CA 94720-3840, USA}
\email{khrystyna.serhiyenko@berkeley.edu}

\begin{abstract} We characterize the indecomposable transjective modules over an arbitrary cluster-tilted algebra that do not lie on a local slice, and we provide a sharp upper bound for the number of (isoclasses of)  these modules.
\end{abstract}

\maketitle
\section{Introduction} 
Cluster-tilted algebras were introduced by Buan, Marsh and Reiten \cite{BMR} and, independently in \cite{CCS} for type $\mathbb{A}$.
%, as a byproduct of the now extensive theory of cluster algebras of Fomin and Zelevinsky \cite{FZ}. Since then, cluster-tilted algebras have been the subject of several investigations, see, for instance,
%\cite{ABS,ABS2, BMR2, KR}. In particular, 
In \cite{ABS} is given a  procedure for constructing cluster-tilted algebras: let $C$ be a triangular algebra of global dimension two over an algebraically closed field $k$, and consider the $C$-$C$-bimodule $\E$, where $D=\Hom_k(-,k)$ is the standard duality, with its natural left and right $C$-actions.
The trivial extension of $C$ by this bimodule is called the {\em relation-extension} of $C$. It is shown there that, if $C$ is tilted, then its relation-extension is cluster-tilted, and every cluster-tilted algebra occurs in this way.

This relation between tilted and cluster-tilted algebras has been studied further in \cite{ABS2}. Inspired by the complete slices in the module categories of tilted algebras, the authors introduced the concept of {\em local slices} as a generalization of complete slices, by relaxing a convexity condition. In \cite{ABS2} it is shown that every cluster-tilted algebra $B$ admits a local slice $\zS$ and that, for every such local slice $\zS$, the quotient algebra $B/\Ann\,\zS$ of $B$ by the annihilator of $\zS$ is a tilted algebra with complete slice $\zS$.
Furthermore, there is a unique component in the Auslander-Reiten quiver of $B$, called the {\em transjective component}, that contains all local slices. Indecomposable modules in this transjective component are called {\em transjective}.

In the module category of a tilted algebra, a complete slice should be thought of as a rather special configuration reproducing the quiver of a hereditary algebra to which our algebra tilts. It is well-known that an algebra is tilted if and only if it admits a complete slice, see, for instance \cite{R}.
In contrast to the above situation, the existence of a local slice does {\em not}  characterize cluster-tilted algebras.
%, and almost every indecomposable transjective module over a cluster-tilted algebra lies on a local slice. 
In \cite{ABS2}, it is shown that if the cluster-tilted algebra is of tree type, then \emph{every} indecomposable transjective module  lies on a local slice. On the other hand, the authors also gave an example of an indecomposable transjective module over a cluster-tilted algebra of type $\widetilde{\mathbb{A}}_{2,1}$ that does {\em not} lie on a local slice.
 
However, the questions `which indecomposable transjective modules do not lie on local slices", and ``how many of these modules do exist", remained open.

It is the purpose of the current paper to answer both questions for arbitrary cluster-tilted algebras. First, we characterize the indecomposable transjective modules that do not lie on a local slice  in Theorem \ref{thm ls}, using the completion of strong sinks defined in \cite{AsScSe}. 
%(generalizing a construction of \cite{ABS4}). 
Then we prove that the number of isoclasses (= isomorphism classes) of indecomposable transjective modules not lying on local slices is finite, and we actually give a sharp bound for this number in Corollary~\ref{cor 2}.

\section{Preliminaries}

\subsection{Notation}
Throughout this paper, algebras are basic and connected finite dimensional algebras over a fixed algebraically closed field $k$.  For an algebra $B$, we denote by $\text{mod}\,B$ the category of finitely generated right $B$-modules.  All subcategories are full, and identified with their object classes.  Given a category $\mathcal{C}$, we sometimes write $M\in\mathcal{C}$ to express that $M$ is an object in $\mathcal{C}$.  
%If $\mathcal{C}$ is a full subcategory of $\text{mod}\,B$, we denote by $\text{add}\,\mathcal{C}$ the full subcategory of $\text{mod}\,B$ having as objects the finite direct sums of summands of modules in $\mathcal{C}$.  

For a point $x$ in the ordinary quiver of a given algebra $B$, we denote by $P(x)$, $I(x)$, $S(x)$ respectively, the indecomposable projective, injective and simple $B$-modules corresponding to $x$.  We denote by $\Gamma(\text{mod}\,B)$ the Auslander-Reiten quiver of $B$ and by $\tau, \tau^{-1} $ the Auslander-Reiten translations.  For further definitions and facts, we refer the reader to \cite{ASS}.

\subsection{Tilting}
Let $Q$ be a finite connected and acyclic quiver.  A module $T$ over the path algebra $kQ$ of $Q$ is called \emph{tilting} if $\text{Ext}^1_{kQ}(T,T)=0$ and the number of isoclasses  of indecomposable summands of $T$ equals $|Q_0|$, see \cite{ASS}.  An algebra $C$ is called \emph{tilted of type $Q$} if there exists a tilting $kQ$-module $T$ such that $C=\text{End}_{kQ} T$.  An algebra $C$ is tilted if and only if it contains a \emph{complete slice} $\Sigma$, see \cite{R}, that is, a finite set of indecomposable modules such that 
\begin{itemize}
\item[1)] $\bigoplus_{U\in \Sigma} U$ is a sincere $C$-module.  
\item[2)] If $U_0\!\to\! U_1\! \to\! \dots \!\to\! U_t$ is a sequence of nonzero morphisms between indecomposable modules with $U_0,U_t\in\Sigma$ then $U_i\in\Sigma$ for all $i$ (\emph{convexity}).
\item[3)]  If $M$ is an indecomposable non-projective $C$-module then at most one of $M$, $\tau M$ belongs to $\Sigma$.
\item[4)]  If  $M,S$ are indecomposable $C$-modules, $f\colon M\rightarrow S$ an irreducible morphism and $S\in\Sigma$, then either $M\in\Sigma$ or $M$ is non-injective and $\tau^{-1} M\in \Sigma$. 
\end{itemize}

%For more on tilting and tilted algebras, we refer the reader to \cite{ASS}.  

\subsection{Cluster-tilted algebras} 
Let $Q$ be a finite, connected and acyclic quiver.  The \emph{cluster category} $\mathcal{C}_Q$ of $Q$ is defined as follows, see \cite{BMRRT}.  
Let $F$ denote the composition $\tau^{-1}_{\mathcal{D}}[1]$, where $\tau^{-1}_{\mathcal{D}}$ denotes the inverse Auslander-Reiten translation in the bounded derived category $\mathcal{D} = \mathcal{D}^b(\text{mod}\, kQ)$, and [1] denotes the shift of $\mathcal{D}$.  Then $\mathcal{C}_Q$ is the orbit category $\mathcal{D}/F$: its objects are the $F$-orbits $\widetilde{X}=(F^i X)_{i\in\mathbb{Z}}$ of the objects $X\in\mathcal{D}$, and the space of morphisms from $\widetilde{X}=(F^i X)_{i\in\mathbb{Z}}$ to $\widetilde{Y}=(F^i Y)_{i\in\mathbb{Z}}$ is 
$\text{Hom}_{\mathcal{C}_Q}(\widetilde{X}, \widetilde{Y}) = \bigoplus_{i\in\mathbb{Z}} \text{Hom}_{\mathcal{D}}(X, F^i Y).$
Then $\mathcal{C}_Q$ is a triangulated category with almost split triangles and, moreover, for $\widetilde{X}, \widetilde{Y}\in\mathcal{C}_Q$ we have a bifunctorial isomorphism $\text{Ext}^1_{\mathcal{C}_Q}(\widetilde{X}, \widetilde{Y})\cong D\text{Ext}^1_{\mathcal{C}_Q}(\widetilde{Y},\widetilde{X})$.  This is expressed by saying that the category $\mathcal{C}_Q$ is \emph{2-Calabi-Yau}.

An object $\widetilde{T}\in\mathcal{C}_Q$ is called \emph{tilting} if $\text{Ext}^1_{\mathcal{C}_Q}(\widetilde{T}, \widetilde{T})=0$ and the number of isoclasses of indecomposable summands of $\widetilde{T}$ equals $|Q_0|$.  The endomorphism algebra $B=\text{End}_{\mathcal{C}_Q} \widetilde{T}$ is then called \emph{cluster-tilted} of type $Q$.  

Let now $T$ be a tilting $kQ$-module, and $C=\text{End}_{kQ} T$ the corresponding tilted algebra.  Then it is shown in \cite{ABS} that the trivial extension $\widetilde{C}$ of $C$ by the $C$-$C$-bimodule $\text{Ext}^2_C (DC,C)$ with the two natural actions of $C$, the so-called \emph{relation-extension} of $C$, is cluster-tilted.  Conversely, if $B$ is cluster-tilted, then there exists a tilted algebra $C$ such that $B=\widetilde{C}$.  

\subsection{Local slices}
Let  $B$ be a cluster-tilted algebra, then a full connected subquiver $\Sigma$ of $\Gamma(\text{mod}\,B)$ is a \emph{local slice}, see \cite{ABS2}, if: 
\begin{itemize}
\item[1)] $\Sigma$ is a \emph{presection}, that is, if $X\to Y$ is an arrow then: 
\begin{itemize}
\item[(a)] $X\in\Sigma$ implies that either $Y\in\Sigma$ or $\tau Y \in \Sigma$
\item[(b)] $Y\in\Sigma$ implies that either $X\in \Sigma$ or $\tau^{-1} X\in\Sigma$.
\end{itemize}
\item[2)] $\Sigma$ is \emph{sectionally convex}, that is, if $X=X_0\to X \to \dots \to X_t = Y$ is a sectional path in $\Gamma(\text{mod}\,B)$ then $X,Y\in\Sigma$ imply that $X_i\in\Sigma$ for all $i$.  
\item[3)] $|\Sigma_0| = \text{rk}\,K_0(B)$.  
\end{itemize}

Let $C$ be tilted, then, under the standard embedding $\text{mod}\,C \to \text{mod}\,\widetilde{C}$, any complete slice in the tilted algebra $C$ embeds as a local slice in $\text{mod}\,\widetilde{C}$, and any local slice in $\textup{mod}\,\Ctilde$ occurs in this way.  If $B$ is a cluster-tilted algebra, then a tilted algebra $C$ is such that $B=\widetilde{C}$ if and only if there exists a local slice $\Sigma$ in $\Gamma(\text{mod}\,B)$ such that $C=B/\text{Ann}_B \Sigma$, where $\text{Ann}_B \Sigma = \bigcap_{X\in\Sigma} \text{Ann}_B X$, see \cite{ABS2}.

\subsection{Completions and reflections}\label{sect 25} We recall the definition of reflections from \cite{AsScSe}. Let $B$ be a cluster-tilted algebra.
Let $\Sigma$ be a local slice in the transjective component of $\Gamma(\text{mod}\,B)$ having the property that all the sources in $\Sigma$ are injective $B$-modules.  Then $\Sigma$ is called a \emph{rightmost} slice of $B$.  Let $x$ be a point in the quiver of $B$ such that $I(x)$ is an injective source of the rightmost slice $\Sigma$.  %Then $x$ is called a \emph{strong sink}.  

 The \emph{completion $H_x$ of $x$} is defined by the following three conditions.
\begin{itemize}
\item [(a)] $I(x)\in H_x$.
\item [(b)] $H_x$ is closed under predecessors in $\zS$.
\item [(c)] If $L\to M$ is an arrow in $\zS$ with $L\in H_x$ having an injective successor in $H_x$ then $M\in H_x$.
\end{itemize}
%\end{definition}
%
The completion $H_x$ can be constructed inductively in the following way. We let $H_1=I(x)$, and $H_2'$ be the closure of $H_1$ with respect to (c). We then let $H_2$ be the closure of $H_2'$ with respect to predecessors in $\zS$. Then we repeat the procedure; given $H_i$, we let $H_{i+1}'$ be the closure of $H_i$ with respect to (c) and $H_{i+1}$ be the closure of $H_{i+1}'$ with respect to predecessors. This procedure must stabilize, because the slice $\zS$ is finite. If $H_j=H_k$ with $k>j$%and $j$ minimal for this property
, we let $H_x=H_j$.

We can decompose $H_x$ as the disjoint union of three sets as follows. Let $\calj$ denote the set of injectives in $H_x$, let $\calj^-$ be the set of non-injectives in $H_x$ which have an injective successor in $H_x$, and let $\cale=H_x\setminus(\calj\cup\calj^-)$ denote the complement of $(\calj\cup\calj^-)$ in $H_x$. Thus
$H_x=\calj\sqcup\calj^-\sqcup\cale$
is a disjoint union.
The {\em reflection of the slice $\zS$ in $x$} is defined as
 \[\zs_x^+\zS=\tau^{-2}(\calj\cup\calj^-)\cup\tau^{-1}\cale\cup(\zS\setminus H_x),\]
 where $\tau^{-2}\calj$ stands for the set of all indecomposable projectives $P(y)$ such that the corresponding injective $I(y)$ is in the set $\mathcal{J}$.

\begin{thm}\cite[Theorem 4.4]{AsScSe}\label{thm reflection}
 Let $\zS$ be a rightmost local slice in $\textup{mod}\,B$  with injective source $I(x)$. Then the reflection $\zs_x^+\zS$ is a local slice as well.
\end{thm}

\section{Main results}
In this section, we prove our main results. We start with two preparatory lemmas.
\begin{definition}
 Let $B$ be a representation-infinite  cluster-tilted algebra and let $\zS,\zS'$ be two local slices in $\textup{mod}\,B$ and $\widetilde{\zS},\widetilde{\zS'}$ be their lifts in the cluster category $\calc$. Then for every indecomposable module $X$ in $\zS$, we define
 $d_X(\zS,\zS')$ to be the unique integer $k$ such that $\tau_\calc^{-k}\widetilde X $ lies in $\widetilde{\zS'}$, where $\widetilde X $ is the lift of $X$ in $\calc$.
\end{definition}
 \begin{remark}
In the above definition, the condition that $B$ is representation-infinite is necessary for the uniqueness of the integer $k$.
\end{remark}

\begin{lemma}\label{lem 4}
 Let $B$ be a representation-infinite cluster-tilted algebra. Let $\zS$ be a rightmost local slice in $ \textup{mod}\,B$ with source $I(x)$, and $H_x$ the completion in $\zS$. Suppose that $\zS'$ is another local slice such that $d_{I(x)}(\zS,\zS') \ge2$. Then for every indecomposable module $Y$ in $H_x$ we have $$d_Y(\zS,\zS')\ge 1.$$ 
 In particular, for every injective indecomposable $I(y)$ in $H_x$ we have $$d_{I(y)}(\zS,\zS')\ge 2.$$
  \end{lemma}

\begin{proof}
 Let $\{I(x)\}=H_1\subset H_2\subset \cdots\subset H_r= H_x$ be the recursive construction of $H_x$ as in section \ref{sect 25} above. Recall that given $H_{i-1}$, the set $H_{i}'$ is the closure of $H_{i-1}$ with respect to condition (c) of the definition of $H_x$, and $H_{i}$ is the closure of $H_{i}'$ under predecessors. 
 Let $Y\in H_i\setminus H_{i-1}$. We will prove the result by induction on $i$. 
 
 If $i=1$ then $Y=I(x)$ and we have $d_Y(\zS,\zS')\ge 2$ by assumption.
 Now assume that $i>1$. Then there are two possibilities
 
 a) Suppose first that $Y\in H_i'$. Then there exists an arrow $L\to Y$ in $\zS$ with $L\in H_{i-1}$ having  an injective successor $I$ in $H_{i-1}$. So there is a path 
 \[\ell:L=L_0\to L_1 \to \cdots \to L_{s-1}\to L_s=I\] in $H_{i-1}$ and our induction hypothesis yields 
  $d_{L_s}(\zS,\zS')=k\ge 2.$
 In the cluster category $\calc$, denote by $\widetilde I$, $\widetilde{L}_i$ and $\widetilde{\zS'}$ the lifts of $I$, $L_i$ and $\zS'$, respectively. Then $\tau^{-k}_\calc \widetilde I \in \widetilde {\zS'}$. Moreover, since $\widetilde{L}_{s-1}\to \widetilde{L}_s$ is an arrow in $\widetilde \zS$, there is an arrow $\tau^{-k}_\calc \widetilde{L}_{s-1}\to\tau^{-k}_\calc \widetilde{L}_s$ in the Auslander-Reiten quiver of $\calc$, and because $\widetilde{\zS'}$ is a local slice, this implies that either 
$\tau^{-k}_\calc \widetilde{L}_{s-1}$ or $\tau^{-(k+1)}_\calc \widetilde{L}_{s-1}$ is in $\widetilde{\zS'}$. In particular
$d_{L_{s-1}}(\zS,\zS')\ge d_{L_{s}}(\zS,\zS')\ge 2.$
Repeating this argument for every arrow in the path $\ell$ we see that 
$d_{L_{i}}(\zS,\zS')\ge 2, \textup{ for all $i$, } $
and thus
$d_{L}(\zS,\zS')\ge 2$. 
This implies that 
$d_{Y}(\zS,\zS')\ge 1$, since there is an arrow $L\to Y$.

b) Now suppose that $Y\in H_i\setminus H_i'$. Thus $Y$ is obtained by closing under predecessors. Hence there is a path 
$\ell': Y=L_0'\to L_1'\to \cdots\to L_t'$
with $L_t'\in H_i'$. In particular, $d_{L'_t}(\zS,\zS')\ge 1$, by part a). By the same argument as in case a), going back along the path $\ell'$ will not decrease the values of the function $d$, so we see that  
$d_{Y}(\zS,\zS')\ge d_{L'_{t}}(\zS,\zS')\ge 1.$ This shows the first claim.
Now, if $Y$ is injective then $d_{Y}(\zS,\zS')$ cannot be equal to 1, because $\tau^{-1} Y=0$ is not in $\zS'$. This shows the second claim.
\end{proof}

For the proof of the next lemma, we need the following construction.
Let $(\Gamma, \tau)$ be a translation quiver, and $X$ be a point in $\Gamma$. Then we define 
%$\Sigma(\to X) $ to consist of all the $Y$ in $\Gamma$ such that there exists a sectional path from $Y$ to $X$ in $\Gamma$, and every path from $Y$ to $X$ is sectional.
%Let $\Lambda$ be a finite dimensional algebra, and $Y$ be an indecomposable $\Lambda$-module in a connected component $\mathcal{T}$ of $\Gamma(\text{mod}\,\Lambda)$.  Let $\Sigma (\to Y)$ and  $\Sigma (Y\to)$ be full subquivers of $\mathcal{T}$ defined by a set of modules below.  
  $${\Sigma (\to X) = \left\{ Y\in\Gamma\left| \begin{array}{c}\textup{ \footnotesize there exists a sectional path from $Y$ to $X$ in $ \zG$}\\ \textup{\footnotesize and every path from $ Y$  to  $X$  in $ \zG $  is sectional.}\end{array}\right.\right\}},$$
 $${\Sigma ( X\to) = \left\{ Y\in\zG         \left|\footnotesize \begin{array}{c}\textup{there exists a sectional path from $X$ to $Y$ in $ \zG$}\\ \textup{and every path from } X \text{ to } Y\text{ in } \zG  \text{  is sectional.}\end{array}\right.\right\}}.$$
%{\small $${\Sigma (\to Y) = \{ X\in\text{ind}\,\Lambda\mid \exists\; X\to \dots \to Y \in \mathcal{T} \text{ and every path from } X \text{ to } Y\text{ in } \mathcal{T} \text{  is sectional}\}}$$}
%{\small $$\Sigma (Y\to) = \{ X\in\text{ind}\,\Lambda\mid \exists\; Y\to \dots \to X \in \mathcal{T} \text{ and every path from } Y \text{ to } X\text{ in } \mathcal{T} \text{  is sectional}\}$$}
%
%Observe that if $\widetilde{Y}$ is an indecomposable object of a cluster category we can define $\Sigma (\to \widetilde{Y})$ and $\Sigma (\widetilde{Y}\to)$ in the corresponding setting in a similar manner as above.  

\begin{prop}\cite[4.2 (6), p. 185]{R} \label{2.6}
Let $Y$ be an indecomposable sincere module in a postprojective or preinjective component.  Then both $\Sigma(\rightarrow Y)$ and $\Sigma(Y\rightarrow)$ are complete slices.  
\end{prop}

\begin{lemma}
 \label{lem 5}
 Let $M$ be an indecomposable transjective $B$-module which does not lie on a local slice. Then there exist an indecomposable injective $B$-module $I(j)$ and a local slice $\Sigma$ containing a sectional path 
  $v: \tau M\to\cdots\to I(j).$
 \end{lemma}
 
 \begin{proof}
 Let $A$ be a hereditary algebra and $T\in\mathcal{C}_A$ a cluster-tilting object such that $B=\text{End}_{\mathcal{C}_A}(T)$.  Let $M$ be an indecomposable $B$-module in the transjective component $\mathcal{T}$ of $\Gamma(\text{mod}\,B)$, and let $\widetilde{M}\in\mathcal{C}_A$ be an indecomposable object such that $\text{Hom}_{\mathcal{C}_A}(T, \widetilde{M})=M$.  Finally, let $\widetilde{\Sigma}=\Sigma(\widetilde{M}\to)$ in the cluster category $\mathcal{C}_A$.  
 
Since $B\cong \text{End}_{\mathcal{C}_A}(\tau_{\mathcal{C}_A}^l T)$ for all $l\in\mathbb{Z}$, we may assume without loss of generality that $\widetilde{\Sigma}$ lies in the postprojective component of $\text{mod}\,A$.  Furthermore, we may assume that every postprojective successor of $\widetilde{\Sigma}$ in $\text{mod}\,A$ in sincere.  Indeed this follows from the fact that there are only finitely many isoclasses of indecomposable postprojective $A$-modules that are not sincere.  For tame algebras this holds, because non-sincere modules are supported on a Dynkin quiver, and for wild algebras see \cite[Corollary 2.3]{Ke}. 

Now since $\widetilde{M}$ is a sincere $A$-module, Proposition \ref{2.6} implies that $\widetilde{\Sigma}$ is a slice in $\text{mod}\,A$, and therefore a local slice in $\mathcal{C}_A$.   Let $\Sigma_1=\text{Hom}_{\mathcal{C}_A}(T, \Sigma(\widetilde{M}\to))$.  Then $M\in\Sigma_1$, and thus by assumption $\Sigma_1$ is not a local slice in $\text{mod}\,B$.  Therefore, there exists an indecomposable direct summand $T_j$ of $T$ such that $\tau T_j\in \Sigma(\widetilde{M}\to)$.  Moreover, by definition of $\Sigma(\widetilde{M}\to)$ there is a sectional path 
 $\widetilde{M}\to\cdots\to \tau T_j $
 and every path from $\widetilde{M}$ to $\tau T_j$ is sectional.  Applying $\tau$ we see that there exists a sectional path 
 $\tilde{v}: \tau \widetilde{M}\to\cdots\to \tau^2 T_j$
 and every path from $\tau\widetilde{M}$ to $\tau^2 T_j$ is sectional.   Thus the local slice $\Sigma(\to \tau^2 T_j)$ in $\mathcal{C}_A$ contains the path $\tilde{v}$.   If there exists a summand $T_i$ of $T$ such that $\tau T_i\in\Sigma(\to \tau^2 T_j)$ then 
 $ 0\not= \text{Hom}_{\mathcal{C}_A}(\tau T_i, \tau^2 T_j)\cong D\text{Ext}^1_{\mathcal{C}_A}(T_j, T_i)$
 which is impossible.  Thus, the local slice $\Sigma(\to \tau^2 T_j)$ does not contain summands of $\tau T$.  Therefore, $\Sigma = \text{Hom}_{\mathcal{C}_A}(T, \Sigma(\to\tau^2 T_j))$ is a local slice in $\text{mod}\,B$ containing $\tau M$ and containing a sectional path 
  $v=\text{Hom}_{\mathcal{C}_A}(T, \tilde{v}): \tau M\to\cdots\to I(j).$
%  This completes the proof of the lemma. 
\end{proof}

We are now ready for our main result.
\begin{thm}\label{thm ls}
 Let $B$ be a cluster-tilted algebra and $M$ an indecomposable transjective $B$-module. Then the following are equivalent.
 \begin{itemize}
\item [\textup{(a)}] 
 $M$ does not lie on a local slice.
\item [\textup{(b)}] There exist a rightmost slice $\zS$ with source $I(x)$  such that  the completion $H_x$ contains a sectional path
 \[\omega :I(x)\to\cdots\to\tau M\to \cdots\to I(j)\] with  $ I(j)$ injective. In particular $\tau M\in \calj^- (H_x)$.
\end{itemize} 
 \end{thm}

\begin{proof} (a)$ \,\Rightarrow \,$(b).
 By Lemma \ref{lem 5}, there is an indecomposable injective $I(j)$ and a local slice $\zS_1$ containing a sectional path 
$v: \tau M\to\cdots\to I(j).$
 Without loss of generality we may assume that there is no other injective on the path $v$ and that $\zS_1$ is a rightmost local slice. Let $u_1:I(x_1)\to\cdots\to \tau M$ be a maximal path in $\zS_1$ ending in $\tau M$. Thus $I(x_1)$ is a source in the rightmost local slice $\zS_1$, hence $I(x_1)$ is injective. Moreover, since $\tau M$ is not an injective module, $ I(x_1) \ne \tau M$.  We distinguish two cases.

(1) If $I(j)\in H_{x_1}$, then  the composition $\omega = u_1v$  lies entirely inside ${ H_{x_1}}$, because ${H_{x_1}}$ is closed under predecessors and we are done.

(2) Now suppose that $I(j)\notin H_{x_1}$.

(2.1) If $\tau M\in H_{x_1}$ then $\tau M$ must lie in $\calj^-$ of $H_{x_1}$, because otherwise the reflection $\zs_{x_1}^+\zS_1$ would be a local slice containing $M$ which is impossible by (a). But $\tau M\in\calj^-$ implies the existence of a path $v':\tau M\to\cdots\to I'$ in $H_{x_1}\subset \zS_1$ with $I'$ injective, and then the path
\[ \omega=u_1v':I(x_1)\to \cdots\to \tau M\to \cdots\to I'\] 
lies entirely inside $H_{x_1}$, and we are done. Note that $\omega $ is sectional since it is a path in a local slice.

(2.2) If $\tau M\notin H_{x_1}$, then the path $v$ lies entirely in $\zS_1\setminus H_{x_1}$ and thus $v$ lies entirely in the local slice $\zS_2=\zs_{x_1}^+\zS_1$.
Repeating the argument, we either obtain a local slice  with source $I(x)$ such that $I(j)\in H_{x}$ and we conclude by the argument of case (1), or we obtain a local slice $\zS_k=\zs_{x_{k-1}}^+\cdots\zs_{x_2}^+\zs_{x_1}^+\zS_1$ containing $v$ and a path $u_k:I(x_k)\to\cdots\to \tau M$
with $I(x_k)$ an injective source and $\tau M\in H_{x_k}$, and we conclude by the   argument of case (2.1).

(b)$ \,\Rightarrow \,$(a).
We want to show that $M$ does not lie on a local slice. Suppose to the contrary that there exists a local slice $\zS_M$ containing $M$. Let $\omega, \zS$ and $H_x$ be as in the statement of the theorem.
% Let \[u:I(x)\to\cdots\to I(i)\] be a maximal path in $\zS$ ending in $I(i)$.
 By the argument of the first part of the proof, we may assume without loss of generality that $I(j)\in H_x$. 
We use the following notation for the path $\omega$
\[I(x)\!\to\!\cdots\!\to\! I(i)\!\to \!L_{\text -s}\!\to\! \cdots\!\to\! L_{\text-2} \!\to\! L_{\text-1}\!\to \!\tau M \!\to \!L_1\!\to\! L_2\!\to\! \cdots\!\to\! L_r\!\to \!I(j),\]
and we assume without loss of generality that none of the $L_i$ is injective. Let $\zg$ be the path obtained by applying $\tau^{-1}$ to { a part of} $\omega $, such that
\[\zg: \tau^{-1} L_{\text-s}\!\to\!\cdots\!\to\!\tau^{-1}L_{\text-2}\!\to\!\tau^{-1} L_{\text-1}\to M \!\to\! \tau^{-1}L_1\!\to\!\tau^{-1} L_2\!\to\! \cdots\!\to\! \tau^{-1}L_r.\]
Since $M$ lies in the local slice $\zS_M$ and $M\to \tau^{-1}L_1$ is an arrow in the Auslander-Reiten quiver, we have that either $\tau^{-1}L_1$ or $L_1$ is in $\zS_M$. If $\tau^{-1}L_1\in \zS_M$ then by the same argument, we have that either $\tau^{-1}L_2 $ or $L_2$ is in $\zS_M$. Repeating this reasoning, we see that either there is an $L_i\in \zS_M$ or $\zS_M$ contains all the $\tau^{-1}L_i$ for $i=1,2,\cdots,r$. In the latter case, we have an arrow $I(j)\to\tau^{-1}L_r$ with $\tau^{-1} L_r\in \zS_M$ and thus $I(j)$ must be in $\zS_M$, since $\tau^{-1}I(j)=0$. Thus in both cases  $\zS_M\cap \omega\ne \emptyset$ and 
\begin{equation}
 \label{Eq 1}
 d_{I(j)}(\zS,\zS_M)\le 0.
\end{equation}
A similar argument along the part of the path $\zg$ from $\tau^{-1}L_{-s}$ to $M$, we see that $\zS_M\cap \tau^{-1}\zg \ne \emptyset$ and $d_{I(i)}(\zS,\zS_M)\ge 2$.
Going back along the { initial segment of the} path $\omega:I(x)\to\cdots\to I(i)$ the values of  the function $d$ cannot decrease, thus $d_{I(x)}(\zS,\zS_M)\ge 2$ as well.
Now using Lemma \ref{lem 4}, we see that $d_{I(j)}(\zS,\zS_M)\ge 2$, which is a contradiction to the inequality (\ref{Eq 1}).
\end{proof}

%\begin{cor}\label{cor 1}
% Let $B$ be a  cluster-tilted algebra and let $M$ be an indecomposable transjective $B$-module.  Then the following are equivalent.
% \begin{itemize}
%\item [\textup{(a)}] 
% $M$ does not lie on a local slice.
%\item [\textup{(b)}] There exist a rightmost slice $\zS$ with source $I(x)$ such that $\tau M \in \mathcal{J}^{-}$ of $H_x$ {\ralf and there is a path $I(x)\to\cdots\to \tau M$ in $\zS$}.  
%\end{itemize} 
%\end{cor}
%
%\begin{proof}
%(a)$ \,\Rightarrow \,$(b). By Theorem \ref{thm ls} there exist a local slice $\Sigma$ with source $I(x)$ such that $\tau M$ and  $\omega$ lie in the completion $H_x$.  In particular, $\tau M$ has an injective successor $I(j)$ in $H_x$, and thus $\tau M \in \mathcal{J}^{-}$ of $H_x$.  
%
%(b)$ \,\Rightarrow \,$(a).  {\ralf There is a path in $H_x$ from $\tau M$ to an injective, because $\tau M\in \mathcal{J}^-$ of $H_x$.   The composition of this  path with the path in condition (b) is sectional because it lies in a local slice, and (a) follows from Theorem \ref{thm ls}.}
%\end{proof}

\begin{remark}\label{rem} For cluster-tilted algebras of tree type,  in particular for rep\-resentation-finite cluster-tilted algebras, we know from \cite{ABS2} that  every indecomposable module  lies on a local slice. 
 Thus condition (b) cannot hold in a cluster-tilted algebra of tree type.
\end{remark}

We now prove that the number of transjective modules over a cluster-tilted algebra which do not lie on a local slice is finite.

\begin{cor}\label{cor 2}
 Let $B$ be a  cluster-tilted algebra.  
Denote by  $n$ the number of isoclasses of indecomposable projective $B$-modules, and define $t$ as the maximum of the number 1 and the number of isoclasses of indecomposable transjective projective $B$-modules. Then the number of isoclasses of indecomposable transjective $B$-modules that do not lie on a local slice is at most 
 \[ (2^{t-1}-1)(n-2).\]
% where $t$ is the number of indecomposable transjective projective $B$-modules and $n$ is the number of all indecomposable projective $B$-modules.  
\end{cor}

\begin{proof}  First observe that if $B$ is representation-finite, then the result trivially holds by Remark \ref{rem}. Assume therefore that $B$ is representation-infinite. By   Theorem~\ref{thm ls} we have that the number of indecomposable transjective $B$-modules that do not lie on a local slice is bounded above by  the cardinality of the set 
$ \cup_{\Sigma} \cup_x \mathcal{J}^-(H_x),$
where $\Sigma$ runs over all rightmost local slices and $x$ { runs over all points such that $I(x) $ is a source} in $\Sigma$.  Since $\mathcal{J}^-(H_x)\subset \Sigma$, we have 
\[ \cup_x\, \mathcal{J}^-(H_x) \subset \{ L\in\Sigma\mid L \text{ is a noninjective indecomposable } B\text{-module}\}\]
and thus 
\[ \left| \cup_x \,\mathcal{J}^-(H_x)\right| \leq n-2\]
because we need at least two injectives in $\Sigma$ for $\mathcal{J}^- (H_x)\not=\emptyset$.   

Let $B=\text{End}_{\mathcal{C}} (T)$ where $T$ is a cluster-tilting object over a cluster category $\mathcal{C}$.   Given a local slice $\Sigma$ in $\text{mod}\,B$ let $\widetilde{\Sigma}$ be the lift of $\Sigma$ to the cluster category, that is $\Sigma = \text{Hom}_{\mathcal{C}}(T, \widetilde{\Sigma})$.  We claim that the number of rightmost local slices $\Sigma$ in $\text{mod}\,B$ is at most $2^{t-1}-1$.  

Observe that for every indecomposable transjective summand $T_i$ of $T$ we have that $\tau T_i$ is a predecessor or a successor of the local slice $\widetilde{\Sigma}$ in $\mathcal{C}$.  Moreover, since the slice $\Sigma$ is rightmost it is determined by the predecessors and successors in $\tau T$ of the corresponding $\widetilde{\Sigma}$.   We have to subtract 1 because if $\widetilde{\Sigma}$ has no transjective successors in $\tau T$ then $\Sigma$ is not rightmost.  This shows that the number of local slices is at most $2^t-1$.  

Finally, for $\mathcal{J}^-(H_x)\not=\emptyset$ there must be at least two summands of $\tau T$ which cannot be separated by a local slice, because   $\mathcal{J}^-(H_x)\not=\emptyset$ implies that there is a sectional path 
$\omega: I(i)\to\cdots\to \tau M \to \cdots \to I(j)$
and in the cluster category this yields a sectional path 
$\tau^{-1}\tilde{\omega}: \tau T_i \to \cdots \widetilde{M} \to \cdots \to \tau T_j$
and $M = \text{Hom}_{\mathcal{C}}(T, \widetilde{M})$ does not lies on a local slice.  This shows that the number of local slices is at most $2^{t-1}-1$.  
\end{proof}

\section{Two examples}
We conclude with two examples. The first example shows that the bound in 
 Corollary \ref{cor 2} is sharp, and the second example illustrates the statement of the theorem.
\begin{example}
%Let $A = \widetilde{\mathbb{A}}_{2,1}$ be the path algebra of the quiver 
%
%\[\xymatrix{1\ar[r] \ar@/^10pt/[rr]& 2 \ar[r] & 3}\]
%
%Let $\mathcal{C}_A$ be the corresponding cluster category and let $T\in\mathcal{C}_A$ be the cluster-tilting object 
%
%\[ T = P(1)\oplus {\begin{array}{c}1\\2\end{array}}\oplus P(3)\]

Let $B$ be the cluster-tilted algebra of type $\widetilde{\mathbb{A}}_{2,1}$ given by the following quiver with relations $\alpha\beta=\beta\gamma=\gamma\alpha=0$.
%\vspace{-5 pt}
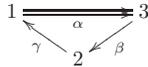
\[\scalebox{0.8}{\xymatrix@R10pt@C20pt{1\ar@<-1pt>[rr]_{\alpha}\ar@<1pt>[rr]&&3\ar[dl]^{\beta}\\
&2\ar[ul]^{\gamma}}}\]
The projective $B$-modules $P(1)$ and $P(3)$ lie in the transjective component of $\Gamma(\text{mod}\,B)$ while the projective $P(2)$ lies in a tube.  
%
%\[\xymatrix{&& I(1)\ar[dr] && \circ && P(1)\\
%& S(1)\ar[ur] && S(2) \ar[dr] && S(3) \ar[ur]\\
%I(3)\ar[ur] \ar@/^15pt/[uurr] &&\circ && P(3) \ar[ur]\ar@/^15pt/[uurr]
%}\]  
%
The only transjective $B$-module not lying on a local slice is $S(2)$.  On the other hand the formula 
in Corollary \ref{cor 2}   
 gives 
$ (2^{t-1}-1)(n-2) = (2^{2-1}-1)(3-2)=1.$
\end{example}

%\subsection{Example}
\begin{example}
We give an example to illustrate Theorem \ref{thm ls}.
%  \begin{example} 
 Let $A$ be the path algebra of the quiver \vspace{-5 pt}
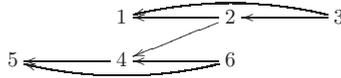
 \[\scalebox{0.8}{\xymatrix@R8pt@C40pt{ 
 &1&2\ar[l]\ar[dl]& 3\ar@/_7pt/[ll]\ar[l]\\
5&4\ar[l]& 6\ar@/^7pt/[ll]\ar[l]\\ }}
 \]
 \vspace{5pt}
 
\noindent Mutating at the vertices 2,4 and 6 yields the cluster-tilted algebra $B$ with quiver
 \vspace{-10 pt}
 
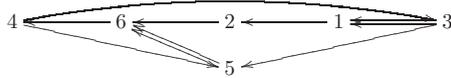
 \[\scalebox{0.8}{\xymatrix@R10pt@C40pt{ 
4\ar[rrd]\ar@/^10pt/[rrrr]&6\ar[l]&2\ar[l]&1\ar[l]&3\ar@<1pt>[l]\ar@<-1pt>[l]\ar[dll]\\
&&5\ar@<1pt>[lu]\ar@<-1pt>[lu]
}}
 \]
 \vspace{5pt}

 \noindent
 In the Auslander-Reiten quiver of $\textup{mod}\,B$ we have the following local configuration.
\vspace{-10 pt}
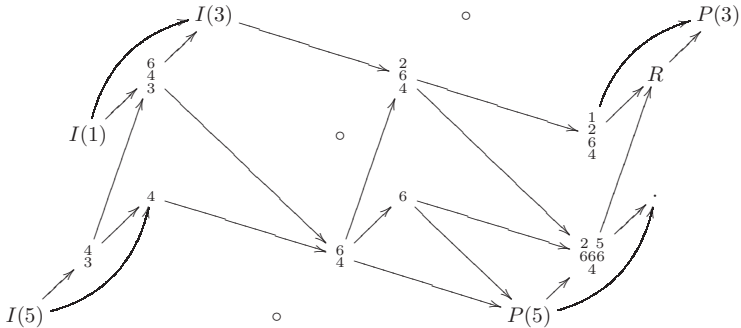
  \[\scalebox{0.8}{\xymatrix@!@R0pt@C1pt{ 
  &&& I(3) \ar[rrrd]
&&&& \circ&&&& P(3)
\\
&& {\begin{smallmatrix} 6\\4\\3 \end{smallmatrix}} \ar[ru]\ar[rrrddd] 
&&&& {\begin{smallmatrix} 2\\6\\4 \end{smallmatrix}} \ar[rrrd]\ar[rrrddd]
&&&&R\ar[ru]
\\
&I(1)\ar[ru]\ar@/^15pt/[rruu]
&&&&\circ 
&&&&{\begin{smallmatrix} 1\\2\\6\\4 \end{smallmatrix}}\ar[ru]\ar@/^15pt/[rruu]
\\
&&{\begin{smallmatrix} 4\end{smallmatrix}}\ar[drrr]
&&&&{\begin{smallmatrix} 6 \end{smallmatrix}}\ar[drrr]\ar[rrdd]
&&&&\cdot
\\
&{\begin{smallmatrix} 4\\3 \end{smallmatrix}}\ar[ru]\ar[ruuu]
&&&&{\begin{smallmatrix} 6\\4 \end{smallmatrix}}\ar[ru]\ar[ruuu]\ar[rrrd]
&&&&{\begin{smallmatrix} 2\ 5\\666\\4 \end{smallmatrix}}\ar[ru]\ar[ruuu]
\\
I(5)\ar[ru]\ar@/_15pt/[rruu]
&&&&\circ
&&&&
P(5)\ar[ru]\ar@/_15pt/[rruu]
}} \]
% \bigskip

%where 
\[\small \begin{array}{ccccccc} I(1)= {\begin{smallmatrix} 2\ \ \,\\ 6\ 6\\4\ 4\\3\ 3\\1 \end{smallmatrix}} & 
I(3)= {\begin{smallmatrix} 2\\6\\4\\3 \end{smallmatrix}}
&
I(5) = 
{\begin{smallmatrix} 4\ \\ \ 3\,4\\\ 5 \end{smallmatrix}}
&
P(5)={\begin{smallmatrix}  5\  \\ 6\,6\ \\ \ 4 \end{smallmatrix}}
&
R={\begin{smallmatrix} 1\quad\\2\,\ 5\\666\\4 \end{smallmatrix}}
&
P(3)= {\begin{smallmatrix}&3&\\1&& 1\quad \\ 2&& 2\,\ 5\\ 6&& 666\\4&& 4 \end{smallmatrix}}
 \end{array}\]

\noindent The 6 modules on the left  form a rightmost local slice 
\[\zS=\{ I(1),  {\begin{smallmatrix} 6\\4\\3 \end{smallmatrix}},I(3),I(5),{\begin{smallmatrix} 4\\3 \end{smallmatrix}},{\begin{smallmatrix} 4 \end{smallmatrix}}\}\] in which both $I(1)$ and $I(5)$ are sources. Their completions are $H_1=\zS$ and $H_5=\{I(5) ,{\begin{smallmatrix} 4\\3 \end{smallmatrix}},{\begin{smallmatrix} 4 \end{smallmatrix}}
\}$.

The module ${\begin{smallmatrix} 6\\4\\3 \end{smallmatrix}}$ satisfies condition (b) of the theorem with respect to $H_1$. Therefore the module $\tau^{-1}{\begin{smallmatrix} 6\\4\\3 \end{smallmatrix}}={\begin{smallmatrix} 2\\6\\4 \end{smallmatrix}}$ does not lie on a local slice.

The module ${\begin{smallmatrix} 4\\3 \end{smallmatrix}}$ does not satisfy condition (b) of the theorem. Indeed, in $H_5$ it does not have an injective successor, and in $H_1$ it is not a successor of $I(1)$. 
 The theorem implies that the module $\tau^{-1}{\begin{smallmatrix} 4\\3 \end{smallmatrix}}={\begin{smallmatrix} 6\\4 \end{smallmatrix}}$ does lie on  a local slice. This local slice is the reflection $\zs_5^+\zS=\{ I(1),  {\begin{smallmatrix} 6\\4\\3 \end{smallmatrix}},I(3),{\begin{smallmatrix} 6\\4 \end{smallmatrix}},{\begin{smallmatrix} 6 \end{smallmatrix}},P(5)\}.$

\end{example}

\bibliographystyle{plain}

\end{document}